\documentclass[a4paper, 11pt]{amsart}
\usepackage{amsmath, amssymb}
\usepackage{amsfonts}
\usepackage{mathrsfs}
\usepackage[arrow,matrix,curve,cmtip,ps]{xy}

\usepackage{amsthm}

\usepackage{color}
\usepackage{amsfonts}
\usepackage{amscd}
\usepackage{comment}
\usepackage{epsfig}

\usepackage[colorlinks=true, pdfstartview=FitV, linkcolor=red, citecolor=blue, urlcolor=green]{hyperref}
\usepackage{tikz,etex,bbm,xspace,hyperref}
\newcommand{\arxiv}[1]{\href{http://arxiv.org/abs/#1}{\tt arXiv:\nolinkurl{#1}}}

\allowdisplaybreaks

\newtheorem{theorem}{Theorem}[section]
\newtheorem{lemma}[theorem]{Lemma}
\newtheorem{proposition}[theorem]{Proposition}

\newtheorem*{theorem*}{Theorem}
\theoremstyle{remark}
\newtheorem{remark}[theorem]{Remark}
\newtheorem{definition}[theorem]{Definition}

\numberwithin{equation}{section}

\newcommand{\ci}[1]{_{ {}_{\scriptstyle #1}}}

\newcounter{vremennyj}

\renewcommand{\eqref}[1]{Equation (\ref{#1})}
         
\begin{document}
\title[Bellman function, Carleson Embedding Theorem and Doob's Inequality]{The Bellman functions of the Carleson Embedding Theorem and the Doob's Martingale Inequality}

\author{Jingguo Lai}

\address{Department of Mathematics \\ Brown University \\ Providence, RI 02912 \\ USA}
\email{jglai@math.brown.edu}





\begin{abstract}
Evaluation of the Bellman functions is a difficult task. The exact Bellman functions of the dyadic Carleson Embedding Theorem \ref{THM1} and the dyadic maximal operators are obtained in \cite{M} and \cite{VV}. Actually, the same Bellman functions also work for the tree-like structure. In this paper, we give a self-complete proof of the coincidence of the Bellman functions on the more general infinitely refining filtered probability space, see Definition \ref{Def2}. The proof depends on a remodeling of the Bellman function of the dyadic Carleson Embedding Theorem.

\end{abstract}

\maketitle

\section{Introduction}
\par Throughout this paper, we denote the Lebesgue measure of a set $E$ by  $|E|$, the average value of $f$ on an interval $I$ by $\langle f\rangle_{\ci{I}}$ and the conjugate H\"{o}lder exponent of $p$ by $p'$, $\frac{1}{p}+\frac{1}{p'} = 1$. The celebrated dyadic Carleson Embedding Theorem states

\begin{theorem}[Dyadic Carleson Embedding Theorem]\label{THM1}
Let $\mathcal{D}=\{([0,1)+j)\cdot2^k: j,k\in\mathbb{Z}\}$ be the standard dyadic lattice on $\mathbb{R}$, and let $\{\alpha_{\ci{I}}\}_{\ci{I\in\mathcal{D}}}$ be a sequence of non-negative numbers satisfying the Carleson condition that: $\sum_{\ci{J\in\mathcal{D},J\subseteq I}}\alpha_{\ci{I}}\leq C |I|$ holds for all dyadic intervals $I\in\mathcal{D}$. Then the embedding
\begin{align}
\sum_{I\in\mathcal{D}}\alpha_{\ci{I}}|\langle f\rangle_{\ci{I}}|^p\leq C_p\cdot C||f||_{\ci{L^p}}^p~holds~for~all~f\in L^p,~where~p>1.\label{eq1}
\end{align}
Moreover, the constant $C_p=(p')^p$ is sharp (cannot be replaced by a smaller one).
\end{theorem}

\par An approach of proving Theorem \ref{THM1} is the introduction of the Bellman function. Without loss of generality, we can assume $f\geq0$. Following \cite{NT} and \cite{NTV}, we define the Bellman function in three variables $(F,\textbf{f},M)$ as
\begin{align}
\mathcal{B}(F,\textbf{f},M; C)=\sup\left\{|I|^{-1}\sum_{J\in\mathcal{D},J\subseteq I}\alpha_{\ci{J}}\langle  f\rangle_{\ci{J}}^p:f,\{\alpha_{\ci{I}}\}_{\ci{I\in\mathcal{D}}}~\textup{satisfy (i), (ii), (iii), and (iv)}\right\},\label{eq2}
\end{align}
$$\textup{(i)}~\langle f^p\rangle_{\ci{I}}=F;~\textup{(ii)}~\langle f\rangle_{\ci{I}}=\textbf{f};~\textup{(iii)}~|I|^{-1}\sum_{J\subseteq{I}}\alpha_{\ci{J}}=M;~\textup{(iv)}~\sum_{J'\subseteq J}\alpha_{\ci{J}}\leq C|J|~\textup{for all}~J\in\mathcal{D}.$$

Note that the Bellman function $\mathcal{B}(F,\textbf{f},M; C)$ defined above dose not depend on the choice of the interval $I$.

\par In \cite{NT}, (\ref{eq1}) was first proved using the Bellman function method for the case $p=2$, and in \cite{NTV}, the sharpness for the case $p=2$ was also claimed. Later, A. Melas found in \cite{M} the exact Bellman function for all $p>1$ in a tree-like setting using combinatorial and covering reasoning. In \cite{VV}, an alternative way of finding the exact Bellman function based on Monge-Amp\`{e}re equation was also established.

\par The Bellman functions have deep connetions to the Stochastic Optimal Control theory \cite{NTV}. Finding the exact Bellman functions is a difficult task. Both the combinatorial methods in \cite{M} and the methods of solving the Bellman PDE in \cite{VV} are quite complicated. Luckily, the proof of Theorem \ref{THM1} only needs a super-solution instead of the exact Bellman function, see \cite{NT}, \cite{NTV} and section 3. The computation of the exact Bellman functions usually reflects deeper structure of the corresponding harmonic analysis problem. It is interesting to note that the exact Bellman function of Theorem \ref{THM1} is not restricted to the standard dyadic lattice. In \cite{M}, it also works for the tree-like structure. Let us here consider a more general situation.

\par Let $(\mathcal{X}, \mathcal{F}, \{\mathcal{F}_n\}_{n\geq0}, \mu)$ be a discrete-time filtered probability space. By a discrete-time filtration, we mean a sequence of non-decreasing $\sigma$-fields
$$\{\emptyset, \mathcal{X}\}=\mathcal{F}_0\subseteq\mathcal{F}_1\subseteq...\subseteq\mathcal{F}_n\subseteq...\subseteq{\mathcal{F}}.$$
We introduce notations $f_n=\mathbb{E}^\mu[f|\mathcal{F}_n]$ and $\langle f\rangle_{\ci{E,\mu}}=\mu(E)^{-1}\int_Efd\mu$.

\begin{definition}\label{Def1}
A sequence of non-negative random variables $\{\alpha_n\}_{n\geq0}$ is called a \emph{Carleson sequence}, if each $\alpha_n$ is $\mathcal{F}_n$-measurable and
\begin{align}\label{eq6}
\mathbb{E}^\mu\left[\sum_{k\geq n}\alpha_k|\mathcal{F}_n\right]\leq C~\textup{for every}~ n\geq0.
\end{align}
\end{definition}

\begin{definition}\label{Def2}
$\{\mathcal{F}_n\}_{n\geq0}$ is called an \emph{infinitely refining filtration}, if for every $\varepsilon>0$, every $n\geq0$ and every set $E\in\mathcal{F}_n$, there exists a real-valued $\mathcal{F}_k$-measurable $(k>n)$ random variable $h$, such that: (i) $|h\textbf{1}_{\ci{E}}|=\textbf{1}_{\ci{E}}$ and (ii) $\int_E|h_n|d\mu\leq\varepsilon$.
\end{definition}

\begin{remark}
The random variable $h$ behaves like a Haar function in classical harmonic analysis. It allows us to reduce the general probabilistic problem to a dyadic one.
\end{remark}

\begin{remark}
The tree-like structure considered in \cite{M} is an infinitely refining filtration. So our results can be seen as a generalization of A. Melas' in this point of view.
\end{remark}

\begin{theorem}[Martingale Carleson Embedding Theorem]\label{THM3} 
If $f\in L^p(\mathcal{X}, \mathcal{F}, \mu)$ and $\{\alpha_n\}_{n\geq0}$ is a Carleson sequence, then
\begin{align}\label{eq7}
\mathbb{E}^\mu\left[\sum_{n\geq0}\alpha_n|f_n|^p\right]\leq C_p\cdot C\cdot\mathbb{E}^\mu\left[|f|^p\right].
\end{align}
Moreover, if $\{\mathcal{F}_n\}_{n\geq0}$ is an infinitely refining filtration, then the constant $C_p=(p')^p$ is sharp.
\end{theorem}

\par Here, again without loss of generality, we can assume $f\geq0$. We define the Bellman function $\mathcal{B}_{\mu}^{\mathcal{F}}(F,\textbf{f},M;C)$ in the martingale setting by
\begin{align}\label{eq8}
\mathcal{B}_{\mu}^{\mathcal{F}}(F,\textbf{f},M;C)
=\sup\left\{\mathbb{E}^\mu\left[\sum_{n\geq0}\alpha_nf_n^p\right]: f,  \{\alpha_n\}_{n\geq0} ~\textup{satisfy (i), (ii), (iii) and (iv)}\right\},
\end{align}
$$\textup{(i)}~ \mathbb{E}^\mu[f^p]=F;~\textup{(ii)}~ \mathbb{E}^\mu[f]=\textbf{f};
~\textup{(iii)}~ \mathbb{E}^\mu\left[\sum_{n\geq0}\alpha_n\right]=M;~\textup{(iv)}~  \{\alpha_n\}_{n\geq0}~ \textup{satisfies (\ref{eq6})}.$$

\par Now, we are ready to state the first main theorem in this paper.

\begin{theorem}[Coincidence of the Bellman functions]\label{THM A}
\begin{align}\label{eq 101}
\mathcal{B}_{\mu}^{\mathcal{F}}(F,\textbf{\textup{f}},M;C) \leq \mathcal{B}(F,\textbf{\textup{f}},M;C).
\end{align}

Moreover, if $\{\mathcal{F}_n\}_{n\geq0}$ is an infinitely refining filtration, then

\begin{align}\label{eq 102A}
\mathcal{B}_{\mu}^{\mathcal{F}}(F,\textbf{\textup{f}},M;C) = \mathcal{B}(F,\textbf{\textup{f}},M;C).
\end{align}
\end{theorem}

\par For the Doob's martingale inequality, recall the definition of the maximal function associated to a discrete-time filtration $\{\mathcal{F}_n\}_{n=0}^\infty$
\begin{align}\label{eq13}
f^*(x)=\sup_{n\geq0}|f_n(x)|.
\end{align}
\begin{theorem}[Doob's Martingale Inequality]\label{THM4}
For every $p>1$ and every $f\in L^p(\mathcal{X},\mathcal{F},\mu)$, we have
\begin{align}\label{eq14}
||f^*||_{\ci{ L^p(\mathcal{X},\mathcal{F},\mu)}}^p\leq (p')^p\cdot ||f||_{\ci{ L^p(\mathcal{X},\mathcal{F},\mu)}}^p.
\end{align}
Moreover, if $\{\mathcal{F}_n\}_{n\geq0}$ is an infinitely refining filtration, then the constant $(p')^p$ is sharp.
\end{theorem}

\par The study of the $L^p$-norm of the maximal function was initiated from the celebrated Doob's martingale inequality, e.g. in \cite{D}. The sharpness of this inequality was shown in \cite{B1} and \cite{B2} if one looks at all martingales. For particular martingales including the dyadic case, see \cite{M} and \cite{W}. Theorem \ref{THM4} covers all these results. 

\par  Assuming $f\geq 0$, we define the Bellman function $\widetilde{\mathcal{B}}_{\mu}^{\mathcal{F}}(F,\textbf{f})$ associated to the Doob's martingale inequality by

\begin{align} \label{eq 102}
\widetilde{\mathcal{B}}_{\mu}^{\mathcal{F}}(F,\textbf{f})=\sup\left\{\mathbb{E}^\mu\left[|f^*|^p\right]:\mathbb{E}^\mu[f^p]=F, ~ \mathbb{E}^\mu[f]=\textbf{f}\right\}.
\end{align}

\par The connection between the Carleson Embedding Theorem and the maximal theory has been known and exploited a lot, e.g. in \cite{NT} and \cite{M}. Using this connection, we give a proof of the second main theorem in this paper.

\begin{theorem}[The Bellman function of the maximal operators]\label{THM B}

\begin{align}\label{eq 1001}
\widetilde{\mathcal{B}}_{\mu}^{\mathcal{F}}(F,\textbf{\textup{f}}) \leq \mathcal{B}_{\mu}^{\mathcal{F}}(F,\textbf{\textup{f}},1;C=1).
\end{align}

Moreover, if $\{\mathcal{F}_n\}_{n\geq0}$ is an infinitely refining filtration, then 

\begin{align}\label{eq 1002}
\widetilde{\mathcal{B}}_{\mu}^{\mathcal{F}}(F,\textbf{\textup{f}}) = \mathcal{B}_{\mu}^{\mathcal{F}}(F,\textbf{\textup{f}},1;C=1).
\end{align}
\end{theorem}

The paper is organized as follows. In section 2, we list some properties of the Bellman function $\mathcal{B}_{\mu}^{\mathcal{F}}(F,\textbf{f},M;C)$ associated to Theorem \ref{THM3} and define the Burkh\"{o}lder's hull. In section 3, we discuss the main inequality of $\mathcal{B}(F,\textbf{f},M;C)$ associated to the dyadic Carleson Embedding Theorem \ref{THM1} and the super-solutions. In section 4, we find explicitly the Burkh\"{o}lder's hull and thus a super-solution of the Theorem \ref{THM1}. In section 5, we give a remodeling of the Bellman function $\mathcal{B}(F,\textbf{f},M;C=1)$ for an infinitely refining filtration, which is central to the rest of the paper. In section 6, we prove the first main Theorem \ref{THM A}. In section 7, we prove the second main Theorem \ref{THM B}.

\section{Properties of $\mathcal{B}_{\mu}^{\mathcal{F}}(F,\textbf{f},M;C)$ and the Burkh\"{o}lder's hull}

\subsection{Properties of $\mathcal{B}_{\mu}^{\mathcal{F}}(F,\textbf{f},M,C)$} 
\par We list some properties of $\mathcal{B}_{\mu}^{\mathcal{F}}(F,\textbf{f},M;C)$. 

\begin{proposition}[Properties of the Bellman function $\mathcal{B}_{\mu}^{\mathcal{F}}(F,\textbf{f},M;C)$]\label{P1}
\end{proposition}
\begin{enumerate}
\item Domain: $\textbf{f}^p\leq F$ and $0\leq M\leq C$. 
\item Range: $0\leq \mathcal{B}_{\mu}^{\mathcal{F}}(F,\textbf{f},M;C)\leq C_p\cdot C\cdot F$. 
\item Homogeneity: $\mathcal{B}_{\mu}^{\mathcal{F}}(t^pF,t\textbf{f},M;C)
=t^p\cdot\mathcal{B}_{\mu}^{\mathcal{F}}(F,\textbf{f},M;C)$ for all $t\geq 0$.
\item Scaling Property: $\mathcal{B}_{\mu}^{\mathcal{F}}(F,\textbf{f},M;C)=C\cdot\mathcal{B}_{\mu}^{\mathcal{F}}(F,\textbf{f},\frac{M}{C};1)$.
\item $\mathcal{B}_{\mu}^{\mathcal{F}}(F,\textbf{f},M;C)\geq\mathcal{B}_{\mu}^{\mathcal{F}}(F,\textbf{f},M-\Delta M;C)+\Delta M\cdot\textbf{f}^p$ for $0\leq \Delta M\leq M$. \\
In particular, $\mathcal{B}_{\mu}^{\mathcal{F}}(F,\textbf{f},M;C)$ is increasing in $M$.
\end{enumerate}

\begin{proof}
(i) follows from the H\"{o}lder's inequality and that $\{\alpha_n\}_{n\geq0}$ is a Carleson sequence. (ii) holds if we assume Theorem \ref{THM3} is true. (iii) and (iv) are obtained directly from definition (\ref{eq8}). We explain (v) in more detail. Choose $f\geq0$ and $\{\alpha_n\}_{n\geq 0}$ that almost give the supremum in the definition (\ref{eq8}), i.e. for small $\varepsilon>0$,
$$\mathbb{E}^\mu\left[\sum_{n\geq 0}\alpha_nf_n^p\right]\geq\mathcal{B}_{\mu}^{\mathcal{F}}(F,\textbf{f},M-\Delta M;C)-\varepsilon,$$
where $\mathbb{E}^\mu\left[f^p\right]=F$, $\mathbb{E}^\mu\left[f\right]=\textbf{f}$, $\mathbb{E}^\mu\left[\sum_{n\geq 0}\alpha_n\right]=M-\Delta M$ and $\mathbb{E}^\mu\left[\sum_{k\geq n}\alpha_k|\mathcal{F}_n\right]\leq C$ for every $n\geq 0$. Since $0\leq M\leq C$, if we increase $\alpha_0$ to $\alpha_0+\Delta M$ then everything is retained except we have now $\mathbb{E}^\mu\left[\sum_{n\geq 0}\alpha_n\right]=M$ and
$$\mathbb{E}^\mu\left[\sum_{n\geq 0}\alpha_nf_n^p\right]\geq\mathcal{B}_{\mu}^{\mathcal{F}}(F,\textbf{f},M-\Delta M;C)-\varepsilon+\Delta M\cdot\textbf{f}^p.$$
Letting $\varepsilon\rightarrow 0$, we obtain $\mathcal{B}_{\mu}^{\mathcal{F}}(F,\textbf{f},M;C)\geq\mathcal{B}_{\mu}^{\mathcal{F}}(F,\textbf{f},M-\Delta M;C)+\Delta M\cdot\textbf{f}^p$.
\end{proof}

\begin{remark}
From (\ref{eq8}), we know that $ \mathcal{B}_{\mu}^{\mathcal{F}}(F,\textbf{f},M;C)$ exists and $0\leq \mathcal{B}_{\mu}^{\mathcal{F}}(F,\textbf{f},M;C)\leq C_p\cdot C\cdot F$ if and only if Theorem \ref{THM3} is true. The sharpness is explained as
\begin{align}\label{eq 103}
\sup_{\textbf{f}^p\leq F, ~0\leq M\leq C} \frac{\mathcal{B}_{\mu}^{\mathcal{F}}(F,\textbf{f},M;C)}{C\cdot F}=(p')^p.
\end{align}
\end{remark}

\begin{remark}\label{R1}
By (iv) of Proposition \ref{P1}, it suffices to consider only $C=1$ for the Bellman function. To suppress notation, we denote $\mathcal{B}_{\mu}^{\mathcal{F}}(F,\textbf{f},M)=\mathcal{B}_{\mu}^{\mathcal{F}}(F,\textbf{f},M;C=1)$.
\end{remark}

\subsection{The Burkh\"{o}lder's hull} Assume that $\mathbb{B}_{\mu}^{\mathcal{F}}(F,\textbf{f},M;C)$ is any function satisfying Proposition \ref{P1}. The domain is $\textbf{\textup{f}}^p\leq F, ~0\leq M\leq C$ and the range is $0\leq \mathbb{B}_{\mu}^{\mathcal{F}}(F,\textbf{f},M)\leq C_p\cdot C\cdot F$. We define the \emph{Burkh\"{o}lder's hull} of $ \mathbb{B}_{\mu}^{\mathcal{F}}(F,\textbf{f},M;C)$ by

\begin{align}\label{eq 104}
u_{\ci{C_p}}^{\mathcal{F}}(\textbf{f}, M;C)=\sup_F\left\{\mathbb{B}_{\mu}^{\mathcal{F}}(F,\textbf{f},M;C)-C_p\cdot C\cdot F\right\},~ \textbf{f}\geq 0,~ 0\leq M\leq C.
\end{align}

\begin{proposition}\label{P2} The {Burkh\"{o}lder's hull} $u_{\ci{C_p}}^{\mathcal{F}}(\textbf{\textup{f}}, M;C)$ satisfies
\begin{enumerate}
\item$-C_p\cdot C\cdot\textbf{\textup{f}}^p\leq u_{\ci{C_p}}^\mathcal{F}(\textbf{\textup{f}}, M;C)\leq 0$; 
\item $u_{\ci{C_p}}^\mathcal{F}(t\textbf{\textup{f}}, M;C)=t^p\cdot u_{\ci{C_p}}^\mathcal{F}(\textbf{\textup{f}}, M;C)$ for all $t\geq 0$.
\item $u_{\ci{C_p}}^\mathcal{F}(\textbf{\textup{f}}, M;C)=C\cdot u_{\ci{C_p}}^\mathcal{F}(\textbf{\textup{f}}, \frac{M}{C};1)$
\end{enumerate}
\end{proposition}

\begin{proof}

For (i), note that $-C_p\cdot C\cdot F\leq  \mathbb{B}_{\mu}^{\mathcal{F}}(F,\textbf{f},M)-C_p\cdot C\cdot F\leq 0$ and $\textbf{f}^p\leq F$, taking the supremum in $F$, we have $-C_p\cdot C\cdot\textbf{f}^p\leq u_{\ci{C_p}}^\mathcal{F}(\textbf{f},M;C)\leq 0$. For (ii) and (iii), taking the supreme in $F$ of Proposition \ref{P1} (iii) and (iv), we prove the results.

\end{proof}

\begin{remark}\label{R2}
By (iii) of Proposition \ref{P2}, it suffices to consider only $C=1$ for the Burkh\"{o}lder's hull. To suppress notation, we denote $u_{\ci{C_p}}^\mathcal{F}(\textbf{\textup{f}}, M)=u_{\ci{C_p}}^\mathcal{F}(\textbf{\textup{f}}, M;C=1)$. We will further discuss the Burkh\"{o}lder's hull in section 3 and section 4.
\end{remark}

\section{The main inequality of $\mathcal{B}(F,\textbf{f},M;C)$ and the super-solutions}

\subsection{The main inequality and the super-solutions}

In this and the next section, we restrict ourselves to the Bellman function $\mathcal{B}(F,\textbf{f},M;C)$ of the dyadic Carleson Embedding Theorem \ref{THM1}. In addition to Proposition \ref{P1}, we also have the following crutial property.

\begin{proposition}[The main inequality]\label{P3}
For all triples $(F,\textbf{\textup{f}},M)$ and $(F_\pm,\textbf{\textup{f}}_\pm,M_\pm)$ belong to the domain $\textbf{\textup{f}}^p\leq F$, $0\leq M\leq C$, with $F=\frac{1}{2}(F_++F_-)$, $\textbf{\textup{f}}=\frac{1}{2}(\textbf{\textup{f}}_++\textbf{\textup{f}}_-)$ and $M=\Delta M+\frac{1}{2}  (M_++M_-)$, where $0\leq \Delta M\leq M$, we have

\begin{align}
\mathcal{B}(F,\textbf{\textup{f}},M;C)\geq\frac{1}{2}\left\{\mathcal{B}(F_+,\textbf{\textup{f}}_+,M_+;C)+\mathcal{B}(F_-,\textbf{\textup{f}}_-,M_-;C)\right\}+\Delta M\cdot\textbf{\textup{f}}^p,\label{eq3}
\end{align}

\end{proposition}

\begin{proof}
\par Split the sum in the definition (\ref{eq2}) of $\mathcal{B}(F,\textbf{f},M)$ into three pieces
$$|I|^{-1}\sum_{J\subseteq I}\alpha_{\ci{J}}\langle f\rangle_{\ci{J}}^p=\frac{1}{2}|I_+|^{-1}\sum_{J\subseteq I_+}\alpha_{\ci{J}}\langle f\rangle_{\ci{J}}^p+\frac{1}{2}|I_-|^{-1}\sum_{J\subseteq I_-}\alpha_{\ci{J}}\langle f\rangle_{\ci{J}}^p+|I|^{-1}\alpha_{\ci{I}}\langle f\rangle_{\ci{I}}^p,$$
where $I_\pm$ means the right and left halves of $I$, respectively.
\par Now, we choose $f^\pm$ on the interval $I_\pm$ that almost give the supremum in the definition (\ref{eq2}) of $\mathcal{B}(F_\pm,\textbf{f}_\pm,M_\pm)$, i.e. for small $\varepsilon>0$,
$$|I_\pm|^{-1}\sum_{J\subseteq I\pm}\alpha_{\ci{J}}\langle f^{\pm}\rangle_{\ci{J}}^p\geq\mathcal{B}(F_\pm,\textbf{f}_\pm,M_\pm;C)-\frac{\varepsilon}{2},$$
and note that $|I|^{-1}\alpha_{\ci{I}}\langle f\rangle_{\ci{I}}^p=\Delta M\cdot\textbf{f}$, we conclude
$$|I|^{-1}\sum_{J\subseteq I}\alpha_{\ci{J}}\langle f\rangle_{\ci{J}}^p\geq\frac{1}{2}\{\mathcal{B}(F_+,\textbf{f}_+,M_+;C)+\mathcal{B}(F_-,\textbf{f}_-,M_-;C)\}-\varepsilon+\Delta M\cdot\textbf{f}^p,$$
which yields exactly (\ref{eq3}).
\end{proof}

\begin{remark}
The main inequality (\ref{eq3}) is much stronger than Proposition \ref{P1} (v) and it is specifically associated to $\mathcal{B}(F,\textbf{f},M;C)$.
\end{remark}
\begin{remark}
\par A function satisfies Propositon \ref{P1} and Proposition \ref{P3} is called a \emph{super-solution}. Such a super-solution needs not to be the exact Bellman function and is much easier to find, see section 4. To distinguish from the exact Bellman function, we denote a super-solution by $\mathbb{B}(F,\textbf{f},M;C)$.
\end{remark}

\par We have seen that the dyadic Carleson Embedding Theorem \ref{THM1} gives rise to a super-solution $\mathbb{B}(F,\textbf{f},M;C)$. On the other hand, to prove (\ref{eq1}) and actually Theorem \ref{THM1}, it suffices to find any super-solution. 
\par Indeed, pick $f\geq 0$ and $\{\alpha_{\ci{I}}\}_{\ci{I\in\mathcal{D}}}$ satisfying the Carleson condition. For every dyadic interval $I\in\mathcal{D}$, let $F_{\ci{I}},\textbf{f}_{\ci{I}},M_{\ci{I}}$ be the corresponding averages
$$F_{\ci{I}}=\langle f^p\rangle_{\ci{I}}, \textbf{f}_{\ci{I}}=\langle f\rangle_{\ci{I}}, M_{\ci{I}}=|I|^{-1}\sum_{J\subseteq I}\alpha_{\ci{J}}.$$
Note that $F_{\ci{I}}=\frac{1}{2}(F_{\ci{I_+}}+F_{\ci{I_-}})$, $\textbf{f}_{\ci{I}}=\frac{1}{2}  (\textbf{f}_{\ci{I_+}}+\textbf{f}_{\ci{I_-}})$ and $M_{\ci{I}}=\Delta M_{\ci{I}}+\frac{1}{2}  (M_{\ci{I_+}}+M_{\ci{I_-}})$, where $0\leq\Delta M_{\ci{I}}=|I|^{-1}\alpha_{\ci{I}}\leq M_{\ci{I}}$. For the interval $I$, the main inequality (\ref{eq3}) implies
$$\alpha_{\ci{I}}\langle f\rangle_{\ci{I}}^p\leq|I|\mathbb{B}(F_{\ci{I}},\textbf{f}_{\ci{I}},M_{\ci{I}};C)-|I_+|\mathbb{B}(F_{\ci{I_+}},\textbf{f}_{\ci{I_+}},M_{\ci{I_+}};C)-|I_-|\mathbb{B}(F_{\ci{I_-}},\textbf{f}_{\ci{I_-}},M_{\ci{I_-}};C).$$
\par Going $n$ levels down, we get the inequality
\begin{align*}
\sum_{J\subseteq I,|J|>2^{-n}|I|}\alpha_{\ci{J}}\langle f\rangle_{\ci{J}}^p
& \leq|I|\mathbb{B}(F_{\ci{I}},\textbf{f}_{\ci{I}},M_{\ci{I}};C)-\sum_{J\subseteq I,|J|=2^{-n}|I|}|J|\mathbb{B}  (F_{\ci{J}},\textbf{f}_{\ci{J}},M_{\ci{J}};C)\\
& \leq|I|\mathbb{B}(F_{\ci{I}},\textbf{f}_{\ci{I}},M_{\ci{I}};C)\leq C_p\cdot C\cdot|I|F_{\ci{I}}=C_p\cdot C\cdot\int_If^p.
\end{align*}
\par Applying the above estimate for the intervals $[-2^n,0)$ and $[0,2^n)$ and taking the limit as $n\rightarrow\infty$, we prove exactly (\ref{eq1}).

\begin{remark}
To prove Theorem \ref{THM1}, all amounts to finding a super-solution $\mathbb{B}(F,\textbf{f},M;C)$. We will see in section 4 that the least possible constant for which $\mathbb{B}(F,\textbf{f},M;C)$ exists is $C_p=(p')^p$.
\end{remark}

\subsection{Further properties of $\mathcal{B}(F,\textbf{f},M;C)$} We start with the following celebrated theorem in convex analysis. We will give a proof for the sake of completeness, for more details, see \cite{NP}.

\begin{theorem}
Let $f:\Omega\rightarrow\mathbb{R}$ be a locally bounded function defined on some convex domain $\Omega\in\mathbb{R}^n$ and $f$ satisfies the midpoint concavity: $f(\frac{x+y}{2})\geq\frac{f(x)+f(y)}{2}$ for all $x,y\in\Omega$. Then $f$ is concave and locally Lipschitz.\label{THM2}
\end{theorem}

\begin{proof}
For concavity: If $f$ is not concave, then there exist two points $a,b\in\Omega$, as well as the line segment connecting them $[a,b]=\{\lambda a+(1-\lambda)b:0\leq\lambda\leq1\}\subseteq\Omega$, such that the function $\varphi(\lambda)=f(\lambda a+(1-\lambda)b)-\lambda f(a)-(1-\lambda)f(b)$ verifies
$$-\infty<C=\inf\{\varphi(\lambda):0\leq\lambda\leq1\}<0.$$
Note that we have used $\Omega$ being convex and $f$ being locally bounded here. Furthermore, $\varphi(0)=\varphi(1)=0$ and a direct computation shows that $\varphi$ is also midpoint concave. Take $0<\delta<-\frac{C}{2}$ and let $0\leq\lambda_0\leq1$, such that $\varphi(\lambda_0)\leq C+\delta$, without loss of generality, further assuming $0<\lambda_0<\frac{1}{2}$, hence we have $\varphi(0)=0$ and $\varphi(2\lambda_0)\geq C$, however
$$\varphi(\lambda_0)\leq C+\delta<\frac{C}{2}=\frac{\varphi(0)+\varphi(2\lambda_0)}{2},~\textup{a contradiction!}$$
\par For locally Lipschitz continuity: Given $a\in\Omega$, we can find a ball $B(a,2r)\subseteq\Omega$ on which $f$ is bounded by a constant $M$. For $x\neq y$ in $B(a,2r)$, put $z=y+(\frac{r}{\alpha})(y-x)$, where $\alpha=||y-x||$. Clearly, $z\in B(a,2r)$. Moreover, since $y=\frac{r}{r+\alpha}x+\frac{\alpha}{r+\alpha}z$, from the concavity of $f$ we infer that $f(y)\geq\frac{r}{r+\alpha}f(x)+\frac{\alpha}{r+\alpha}f(z)$. So $|f(y)-f(x)|\leq\frac{\alpha}{r+\alpha}|f(z)-f(x)|\leq\frac{||y-x||}{r}\cdot 2M$.
\end{proof}

\par In the case of our main inequality (\ref{eq3}), first put $F=\frac{1}{2}(F_++F_-)$, $\textbf{f}=\frac{1}{2}(\textbf{f}_++\textbf{f}_-)$ and $M=\frac{1}{2}(M_++M_-)$ (i.e. $\Delta M=0$) and assume all triples $(F,\textbf{f},M)$, $(F_\pm,\textbf{f}_\pm,M_\pm)$ are in the convex domain: $\textbf{f}^p\leq F, 0\leq M\leq C$, then we obtain the midpoint concavity of $\mathcal{B}(F,\textbf{f},M;C)$.
Apply Theorem \ref{THM2} to the function $\mathcal{B}$, so $\mathcal{B}$ is itself concave and locally Lipschitz. In particular, $\mathcal{B}$ is a continuous function.

\par Now let $0\leq \lambda\leq 1$ and $F=\lambda F_+(1-\lambda)F_-$, $\textbf{f}=\lambda\textbf{f}_+(1-\lambda)\textbf{f}_-$, $M=\Delta M+\lambda M_++(1-\lambda)M_-$. Proposition \ref{P1} (v) and concavity of $\mathcal{B}$ imply that

\begin{align*}
\Delta M\cdot\textbf{f}^p
& \leq\mathcal{B}(F,\textbf{f},M;C)-\mathcal{B}(F,\textbf{f},M-\Delta M;C)\\
& \leq\mathcal{B}(F,\textbf{f},M;C)-\left\{\lambda\mathcal{B}(F_+,\textbf{f}_+,M_+;C)+(1-\lambda)\mathcal{B}(F_-,\textbf{f}_-,M_-;C)\right\}.
\end{align*}

Hence, the Bellman function $\mathcal{B}(F,\textbf{f},M)$ is continuous and

\begin{align}\label{eq 110}
\mathcal{B}(F,\textbf{f},M;C)\geq \lambda\mathcal{B}(F_+,\textbf{f}_+,M_+;C)+(1-\lambda)\mathcal{B}(F_-,\textbf{f}_-,M_-;C)+\Delta M\cdot\textbf{f}^p.
\end{align}

\subsection{Regularization of the super-solutions}

\par As we have seen, the Bellman function $\mathcal{B}$ is concave and locally Lipschitz, and thus continuous, but hardly any better than that. Fortunately, we know that the proof of Theorem \ref{THM1} boils down to finding just a super-solution $\mathbb{B}$. We recall the trick of regularization of the super-solutions from \cite{N}.

\par Given a super-solution $\mathbb{B}(F,\textbf{f},M;C)$ satisfying Proposition \ref{P1} and Proposition \ref{P3}. Let $\phi_\varepsilon, \psi_\varepsilon:(0,\infty)\rightarrow[0,\infty)$ be any two nonnegative $C^\infty$ functions, such that supp$(\phi_\varepsilon)\subseteq[1,(1+\varepsilon)^p]$,  supp$(\psi_\varepsilon)\subseteq[1+\varepsilon,1+2\varepsilon]$ and $\int_0^\infty\phi_\varepsilon(t)\frac{dt}{t}=\int_0^\infty\psi_\varepsilon(t)\frac{dt}{t}=1$. Define
\begin{align*}
\mathbb{B}_\varepsilon(F,\textbf{f},M;C)
& = \iiint_{(0,\infty)^3}\mathbb{B}\left(\frac{F}{u},\frac{\textbf{f}}{v},\frac{M}{w};C\right)\phi_\varepsilon(u)\psi_\varepsilon(v)\phi_\varepsilon(w)\frac{dudvdw}{uvw}\\
& = \iiint_{(0,\infty)^3}\mathbb{B}(u,v,w;C)\phi_\varepsilon\left(\frac{F}{u}\right)\psi_\varepsilon\left(\frac{\textbf{f}}{v}\right)\phi_\varepsilon\left(\frac{M}{w}\right)\frac{dudvdw}{uvw}
\end{align*}
Note that the second representation shows $\mathbb{B}_\varepsilon\in C^\infty$. Since $\mathbb{B}$ is continuous, the family of smooth functions $\{\mathbb{B}_\varepsilon:\varepsilon>0\}$ converges to $\mathbb{B}$ pointwisely as $\varepsilon\rightarrow 0$.

\par To check Proposition \ref{P1} and Proposition \ref{P3} for $\mathbb{B}_\varepsilon$. Note that the supports of $\phi_\varepsilon$ and $\psi_\varepsilon$ guarantee that $\mathbb{B}_\varepsilon$ is well-defined in the region $\{\textbf{f}^p\leq F, 0\leq M\leq C\}$ and an easy calculation shows that $0\leq\mathbb{B}_\varepsilon\leq C_p\cdot C\cdot F$. Homogeniety and the scaling property are inherited from $\mathbb{B}$. For the main inequality, the first representation and (\ref{eq3}) imply that
\begin{align*}
\mathbb{B}_\varepsilon(F,\textbf{f},M;C)
& - \frac{1}{2}\left\{\mathbb{B}_\varepsilon(F_+,\textbf{f}_+,M_+;C)+\mathbb{B}_\varepsilon(F_-,\textbf{f}_-,M_-;C)\right\}\\
& \geq\Delta M\cdot\textbf{f}^p\int_1^{(1+\varepsilon)^p}\int_{1+\varepsilon}^{1+2\varepsilon}
\int_1^{(1+\varepsilon)^p}\frac{1}{v^pw}\phi_\varepsilon(u)\psi_\varepsilon(v)\phi_\varepsilon(w)\frac{dudvdw}{uvw}\\
& \geq\frac{1}{(1+2\varepsilon)^p(1+\varepsilon)^p}\Delta M\cdot\textbf{f}^p \rightarrow \Delta M\cdot\textbf{f}^p~\textup{as}~\varepsilon \rightarrow 0.
\end{align*}

\par Hence, the proof of (\ref{eq1}) given in subsection 3.1 works via the smooth function $\mathbb{B}_\varepsilon(F,\textbf{f},M;C)$ as well. In what follows, it suffices to consider only for smooth super-solutions $\mathbb{B}(F,\textbf{f},M;C)$.

\subsection{The main inequality in its infinitesimal version} For a smooth super-solution $\mathbb{B}(F,\textbf{f},M;C)$, being concave means the second differential $d^2\mathbb{B}\leq 0$. By Proposition \ref{P1} (v), we have: $\mathbb{B}(F,\textbf{f},M)-\mathbb{B}(F,\textbf{f},M-\Delta M)\geq\Delta M\cdot\textbf{f}^p$, thus $\frac{\partial\mathbb{B}}{\partial M}\geq\textbf{f}^p$.
\par Therefore, the main inequality (\ref{eq3}) implies the following two infinitesimal ones
\begin{align}
d^2\mathbb{B}(F,\textbf{f},M;C)\leq0~\ \ \ \textup{and}\ \ \ \ ~\frac{\partial\mathbb{B}}{\partial M}(F,\textbf{f},M;C)\geq\textbf{f}^p.\label{eq4}
\end{align}

Actually, (\ref{eq4}) is equivalent to the main inequality (\ref{eq3}). Since by (\ref{eq4}), we can deduce 

\begin{align*}
\Delta M\cdot\textbf{f}^p
& \leq\mathbb{B}(F,\textbf{f},M;C)-\mathbb{B}(F,\textbf{f},M-\Delta M;C)\\
& \leq \mathbb{B}(F,\textbf{f},M;C)-\frac{1}{2}\left\{\mathbb{B}(F_+,\textbf{f}_+,M_+;C)+\mathbb{B}(F_-,\textbf{f}_-,M_-;C)\right\}.
\end{align*}

\section{Finding a super-solution via the Burkh\"{o}lder's hull}

\subsection{Burkh\"{o}lder's hull and some reductions} 
Assume $\mathbb{B}(F,\textbf{\textup{f}},M;C)$ is a smooth super-solution. In this section, we present an explicit function $\mathbb{B}(F,\textbf{\textup{f}},M;C)$ with the help of the Burkh\"{o}lder's hull. Recall from (\ref{eq 104}), define  $u(\textbf{\textup{f}},M;C)=\sup_{\ci{F}}\{\mathbb{B}(F,\textbf{f},M;C)-C_p\cdot C\cdot F\}$. From Remark \ref{R1} and Remark \ref{R2}, it suffices to consider only $C=1$. We use the notation $\mathbb{B}(F,\textbf{\textup{f}},M)=\mathbb{B}(F,\textbf{\textup{f}},M;C=1)$ and $u(\textbf{\textup{f}},M)=u(\textbf{\textup{f}},M;C=1)$ in this section.
\begin{proposition}\label{P4}
The Bukh\"{o}lder's hull $u(\textbf{\textup{f}},M)$ satisfies the following properties:
$$\textup{(i)}~\frac{\partial u}{\partial M}(\textbf{\textup{f}},M)\geq\textbf{\textup{f}}^p~\ \ \ \textup{and\ \ \ \      (ii)}~u(\textbf{\textup{f}},M)~\textup{is concave}.$$
\end{proposition}

\begin{proof}
\begin{enumerate}

\item From $\frac{\partial\mathbb{B}}{\partial M}(F,\textbf{f},M)\geq\textbf{f}^p$ we conclude $\mathbb{B}(F,\textbf{f},M+\Delta M)-\mathbb{B}(F,\textbf{f},M)\geq\Delta M\cdot\textbf{f}^p$. Choose $F_0$ that almost gives the supremum in the definition of $u(\textbf{f},M)$, i.e. for small $\varepsilon>0$, $\mathbb{B}(F_0,\textbf{f},M)-C_p\cdot F_0>u(\textbf{f},M)-\varepsilon$, then
\begin{align*}
u(\textbf{f},M+\Delta M)-u(\textbf{f},M)
& \geq \left[\mathbb{B}(F_0,\textbf{f},M+\Delta M)-C_p\cdot F_0\right]-[\mathbb{B}(F_0,\textbf{f},M)-C_p\cdot F_0+\varepsilon]\\
& = [\mathbb{B}(F_0,\textbf{f},M+\Delta M)-\mathbb{B}(F_0,\textbf{f},M)]-\varepsilon\geq\Delta M\cdot\textbf{f}^p-\varepsilon.
\end{align*}
Letting $\varepsilon\rightarrow0$, so $\frac{\partial u}{\partial M}(\textbf{f},M)\geq\textbf{f}^p$.
\item This follows from a simple lemma.

\begin{lemma}
Let $\varphi(x,y)$ be a convex function and let $\Phi(x)=\sup_{\ci{y}}\varphi(x,y)$, then $\Phi(x)$ is also a concave function.
\end{lemma}

\noindent \emph{Proof}.  We need to see $\Phi(\lambda x_1+(1-\lambda)x_2)\geq\lambda\Phi(x_1)+(1-\lambda)\Phi(x_2)$ for all $x_1, x_2$ and $0\leq\lambda\leq1$. Again choose $y_1$ and $y_2$ in the definition of $\Phi(x)$, such that for small $\varepsilon>0$, $\varphi(x_1,y_1)>\Phi(x_1)-\varepsilon$ and $\varphi(x_2,y_2)>\Phi(x_2)-\varepsilon$. Then
\begin{align*}
\lambda\Phi(x_1)+(1-\lambda)\Phi(x_2)
& < \lambda\varphi(x_1,y_1)+(1-\lambda)\varphi(x_2,y_2)+\varepsilon\\
& \leq\varphi(\lambda x_1+(1-\lambda)x_2,\lambda y_1+(1-\lambda)y_2)+\varepsilon\\
& \leq\Phi((\lambda x_1+(1-\lambda)x_2)+\varepsilon,
\end{align*}
which proves the lemma.
\end{enumerate}
\end{proof}

From Proposition \ref{P2} and Proposition \ref{P4}, if the dyadic Carleson Embedding Theorem \ref{THM1} holds with constant $C_p$, then there exists a concave function $u(\textbf{f},M)$ satisfying  $-C_p\cdot\textbf{f}^p\leq u(\textbf{f},M)\leq0$,  $\frac{\partial u}{\partial M}(\textbf{f},M)\geq\textbf{f}^p$ and $u(t\textbf{f},M)=t^p\cdot u(\textbf{f},M)$, for all $t\geq 0$.

\par On the other hand, if such a $u(\textbf{f},M)$ exists, then we can define $\mathbb{B}(F,\textbf{f},M)=u(\textbf{f},M)+C_p\cdot F$ for $F\geq\textbf{f}^p, 0\leq M\leq 1$, and so $\mathbb{B}$ is a super-solution that proves the dyadic Carleson Embedding Theorem with the same constant $C_p$. Hence, the best constant in the dyadic Carleson Embedding Theorem is exactly the best constant for which the fuction $u(\textbf{f},M)$ exists. 

\par Now using the homogeniety property $u(t\textbf{f},M)=t^p\cdot u(\textbf{f},M)$, $u(\textbf{f},M)$ can be represented as $u(\textbf{f},M)=\textbf{f}^p\cdot\varphi(M)$. For such a function $u(\textbf{f},M)$, the Hessian equals 
\begin{align*}
\begin{pmatrix}
p(p-1)\textbf{f}^{p-2}\varphi(M) & p\textbf{f}^{p-1}\varphi'(M)\\
 p\textbf{f}^{p-1}\varphi'(M) & \textbf{f}^p\varphi{''}(M)
\end{pmatrix},
\end{align*}
so the concavity of $u(\textbf{f},M)$ is equivalent to the following two inequalities
$$\varphi(M)\leq0~\textup{and}~\varphi\varphi^{''}-(p')(\varphi')^2\geq0~\textup{for}~0\leq M\leq1.$$
The inequality $\frac{\partial u}{\partial M}(\textbf{f},M)\geq\textbf{f}^p$ means $\varphi'(M)\geq1$ and $\varphi(M)$ also satisfies $-C_p\leq\varphi(M)\leq0$.\\
Hence, our task is to find $\varphi(M)$, such that
\begin{enumerate}
\item $0\leq M\leq1$
\item $-C_p\leq\varphi(M)\leq0$
\item $\varphi'(M)\geq1$
\item $\varphi\varphi^{''}-(p')(\varphi')^2\geq0$,
\end{enumerate}
and the least possible constant is $C_p=\inf_{\ci{\varphi}}\sup_{\ci{0\leq M\leq1}}\{-\varphi(M)\}$.

\subsection{An explicit super-solution $\mathbb{B}(F,\textbf{f},M)$}
We first introduce $\phi(M)=-\varphi(M)\geq0$, then $\phi(M)$ satisfies
$$\textup{(i)}~0\leq M\leq1;~\textup{(ii)}~0\leq\phi(M)\leq C_p;~\textup{(iii)}~\phi'(M)\leq-1;~\textup{(iv)}~\phi\phi^{''}-(p')(\phi)^2\geq0,$$
and we need to consider $C_p=\inf_{\ci{\phi}}\sup_{\ci{0\leq M\leq1}}\{\phi(M)\}$.
\par Rewrite $\phi\phi^{''}-(p')(\phi)^2\geq0$ as $\phi^{p'+1}\cdot\left(\frac{\phi'}{\phi^{p'}}\right)'\geq0$ or equivalently $\left(\frac{\phi'}{\phi^{p'}}\right)'\geq0$. Let $\left(\frac{\phi'}{\phi^{p'}}\right)'=g(M)\geq0$ and denote $G(M)=\int_0^Mg$, we can solve
$$\phi(M)=\left[\frac{p-1}{C_2M+C_1-\int_0^MG}\right]^{p-1},$$
where $C_1$ and $C_2$ are some constants, such that $C_2M+C_1-\int_0^MG\geq0$ for $0\leq M\leq1$.
\par Note that $\phi'(M)\leq-1$, so $\sup_{\ci{0\leq M\leq1}}\phi(M)=\phi(0)=\left[\frac{p-1}{C_1}\right]^{p-1}$. All we need to do now is to infimize $\left[\frac{p-1}{C_1}\right]^{p-1}$ among all possible $\phi(M)$. 
\par To this end, we compute
$$\phi'(M)=-\left[\frac{p-1}{C_2M+C_1-\int_0^MG}\right]^p\cdot\left[C_2-G(M)\right],$$
and use again $\phi'(M)\leq-1$ with $M=1$, which yields
$$C_1\leq-C_2+\int_0^1G+(p-1)\cdot\left[C_2-G(1)\right]^{\frac{1}{p}}.$$
Remember that $G'(M)=g(M)\geq0$, thus $G(M)$ is increasing, in particular, $\int_0^1G\leq G(1)$, so $C_1\leq-\left[C_2-G(1)\right]+(p-1)\cdot\left[C_2-G(1)\right]^{\frac{1}{p}}$. An easy calculation gives the maximum of the right hand side equals $(p-1)\cdot(p')^{-p'}$ when $C_2=G(1)+(p')^{-p'}$, therefore, $C_1$ is at most $(p-1)\cdot(p')^{-p'}$ and thus  $\left[\frac{p-1}{C_1}\right]^{p-1}\geq(p')^p$.

\par To write down an explicit super-solution, simply take $G(M)=0$, $C_2=(p')^{-p'}$ and $C_1=(p-1)\cdot(p')^{-p'}$, then
$$\phi(M)=\left[\frac{p-1}{C_2M+C_1-\int_0^MG}\right]^{p-1}=\frac{p^p}{(p-1)\cdot[M+(p-1)]^{p-1}},$$
and recall the relation $\mathbb{B}(F,\textbf{f},M)=u(\textbf{f},M)+C_pF=(p')^pF-\textbf{f}^p\cdot\phi(M)$, we obtain
\begin{align}
u(\textbf{f},M)=-\frac{(p\textbf{f})^p}{(p-1)\cdot[M+(p-1)]^{p-1}},
\end{align}

\begin{align}
\mathbb{B}(F,\textbf{f},M)=(p')^pF-\frac{(p\textbf{f})^p}{(p-1)\cdot[M+(p-1)]^{p-1}}.\label{eq5}
\end{align}
In the general case, we have $u(\textbf{\textup{f}},M;C)=C\cdot u(\textbf{\textup{f}},\frac{M}{C})$ and $\mathbb{B}(F,\textbf{\textup{f}},M;C)=C\cdot \mathbb{B}(F,\textbf{\textup{f}},\frac{M}{C})$. This super-solution $\mathbb{B}(F,\textbf{f},M;C)$ gives exactly the sharpness of $C_p=(p')^p$.

\begin{remark}
Now that the dyadic Carleson Embedding Theorem \ref{THM1} is proved, the Bellman function $\mathcal{B}(F,\textbf{f},M;C)$ exists with $C_p=(p')^p$. However, the super-solution $\mathbb{B}(F,\textbf{\textup{f}},M;C)$ obtained above is not the real Bellman function, since on the boundary $F=\textup{f}^p$ the real Bellman function must satisfy the boundary condition $\mathcal{B}(F,\textbf{f},M;C)=M\textbf{f}^p=MF$, but the function we constructed does not satisfy this condition. So, this super-solution only touches the real one along some set. For the exact Bellman function $\mathcal{B}(F,\textbf{f},M;C)$, see \cite{M} and \cite{VV}.
\end{remark}

\section{Remodeling of the Bellman function $\mathcal{B}(F,\textbf{f},M;C=1)$ for an infinitely refining filtration}

\par In this section, we present a remodeling of the Bellman function $\mathcal{B}(F,\textbf{f},M;C=1)$ for an infinitely refining filtration, which is central to the proof of Theorem \ref{THM A} and Theorem \ref{THM B}. We use the notation $\mathcal{B}(F,\textbf{f},M)=\mathcal{B}(F,\textbf{f},M;C=1)$ in this and later sections.
\par Consider the unit interval $I=[0,1]\in\mathcal{D}$, let $\{I_j^k: 1\leq j\leq2^k\}$ be its $k$-th generation decendant by subdividing $I$ into $2^k$ congruent dyadic intervals and denote $I_1^0=I$.
\par Starting from the definition (\ref{eq2}) of the Bellman function $\mathcal{B}(F,\textbf{f},M)$, we can find a function $f\geq 0$ with $\langle f^p\rangle_{\ci{I}}=F$, $\langle f\rangle_{\ci{I}}=\textbf{f}$ and a sequence $\{\alpha_{\ci{J}}\}_{\ci{J\subseteq I}}$, $\sum_{\ci{J\subseteq{I}}}\alpha_{\ci{J}}=M$ satisfying the Carleson condition with constant $C=1$, such that the sum $\sum_{\ci{J\subseteq I}}\alpha_{\ci{J}}\langle f\rangle_{\ci{J}}^p$ (almost) attains $\mathcal{B}(F,\textbf{f},M)$. 
\par To proceed, we further assume that the sequence $\{\alpha_{\ci{J}}\}_{\ci{J\subseteq I}}$ has only \emph{finitely many} non-zero terms. Hence, the indices of $\{\alpha_{\ci{J}}\}_{\ci{J\subseteq I}}$ belong to the collection $\{I_j^k: 1\leq k\leq N, 1\leq j\leq2^k\}$ for some fixed integer $N$, i.e. for all $J\notin\{I_j^k: 1\leq k\leq N, 1\leq j\leq2^k\}$, we have $\alpha_{\ci{J}}=0$. As a consequence, we can think the function $f$ being piecewise constant on all $\{I_j^N:1\leq j\leq2^N\}$.

\par Now, let us do the remodeling. Fix a small $\varepsilon$, $0<\varepsilon<1$. Consider a discrete-time filtered probability space $(\mathcal{X}, \mathcal{F}, \{\mathcal{F}_n\}_{n\geq0}, \mu)$. The initial construction is $\mathcal{X}_1^0=\mathcal{X}$, and this is $\mathcal{F}_{n_0}$-measurable, where $n_0=0$. Assume that the $\mathcal{F}_{n_k}$-measurable sets $\mathcal{X}_j^k, 1\leq j\leq2^k$ are constructed. We want to inductively construct $\mathcal{F}_{n_{k+1}}$-measurable sets $\mathcal{X}_j^{k+1}, 1\leq j\leq2^{k+1}$. Take a $\mathcal{F}_{n_k}$-measurable set $\mathcal{X}_j^k$. Our construction consists two steps.
\par The first step is a modification of the set $\mathcal{X}_j^k$. For the given $\varepsilon>0$ and $\mathcal{X}_j^k\in\mathcal{F}_{n_k}$, Definition \ref{Def2} guarantees the existence of a real-valued $\mathcal{F}_{n_j^k}$-measurable random variable $h$ $(n_j^k>n_k)$, such that: (i) $|h\textbf{1}_{\ci{E}}|=\textbf{1}_{\ci{E}}$ and (ii) $\int_{\mathcal{X}_j^k}|h_{n_k}|d\mu\leq\frac{\varepsilon^2}{4}\mu(X_j^k)$. The condition (ii) is chosen in such a way that
\begin{align}\label{eq9}
\mu\left(\left\{x\in\mathcal{X}_j^k: |h_{n_k}|>\frac{\varepsilon}{2}\right\}\right)\leq\frac{\varepsilon}{2}\mu(\mathcal{X}_j^k).
\end{align}
Let $\widetilde{\mathcal{X}_j^k}=\mathcal{X}_j^k\setminus\left\{x\in\mathcal{X}_j^k: |h_{n_k}|>\frac{\varepsilon}{2}\right\}$. So we can conclude $|h_{n_k}|\leq\frac{\varepsilon}{2}$ on $\widetilde{\mathcal{X}_j^k}$, and moreover, $\left(1-\frac{\varepsilon}{2}\right)\mu(\mathcal{X}_j^k)\leq\mu(\widetilde{\mathcal{X}_j^k})\leq\mu(\mathcal{X}_j^k)$.
\par In the second step, we set $\mathcal{X}_{2j-1}^{k+1}=\widetilde{{X}_j^k}\cap\{h=1\}$ and $\mathcal{X}_{2j}^{k+1}=\widetilde{{X}_j^k}\cap\{h=-1\}$. Since $\left|\int_{\widetilde{\mathcal{X}_j^k}}hd\mu\right|\leq\int_{\widetilde{\mathcal{X}_j^k}}|h_{n_k}|d\mu\leq\frac{\varepsilon}{2}\mu(\widetilde{\mathcal{X}_j^k})$, which gives $\left|\mu(\mathcal{X}_{2j-1}^{k+1})-\mu(\mathcal{X}_{2j}^{k+1})\right|\leq\frac{\varepsilon}{2}\mu({\widetilde{\mathcal{X}_j^k}})\leq\frac{\varepsilon}{2}\mu({\mathcal{X}_j^k})$, we have 
\begin{align}\label{eq10}
\frac{1}{2}(1-\varepsilon)\leq\max\left\{\frac{\mu(\mathcal{X}_{2j-1}^{k+1})}{\mu(\mathcal{X}_j^k)}, \frac{\mu(\mathcal{X}_{2j}^{k+1})}{\mu(\mathcal{X}_j^k)}\right\}\leq\frac{1}{2}(1+\varepsilon).
\end{align}
\par Do this for all $\mathcal{X}_j^k$, $1\leq j\leq2^k$ and let $n_{k+1}=\max\{n_j^k: 1\leq j\leq2^k\}$. Hence, we construct $\mathcal{F}_{n_{k+1}}$-measurable sets $\mathcal{X}_j^{k+1}, 1\leq j\leq2^{k+1}$. Our construction stops when $k=N$.
\par Now that we have constructed $\{\mathcal{X}_j^k: 0\leq k\leq N, 1\leq j\leq2^k\}$. We can define a new sequence  $\{\alpha_n\}_{n\geq0}$ on the space  $(\mathcal{X},\mathcal{F},\mu)$ as 
\begin{align*}
\alpha_{n_k} &=
\left\{ \begin{array}{cc}
\mu(\mathcal{X}_j^k)^{-1}\alpha_{\ci{I_j^k}}, ~\textup{if}~x\in\mathcal{X}_j^k\\
0, ~\textup{if}~ x\in\mathcal{X}\setminus\bigcup_{j=1}^{2^k}\mathcal{X}_j^k\\
\end{array}
\right.
\end{align*}
and $\alpha_n=0$ for all $n$'s different from $n_k, 1\leq k\leq N$.
\par Finally, set the new function $\widetilde{f}$ as $\widetilde{f}\textbf{1}_{\ci{\mathcal{X}_j^N}}=f\textbf{1}_{\ci{I_j^N}}, 1\leq j\leq 2^N$, and set $\widetilde{f}=0$ on $\mathcal{X}\setminus\bigcup_{j=1}^{2^N}\mathcal{X}_j^N$. Note that the function $\widetilde{f}$ is also piecewise constant on all $\{\mathcal{X}_j^k: 0\leq k\leq N, 1\leq j\leq2^k\}$. 

\begin{remark}
This construction guarantees that $\mathbb{E}^\mu\left[\sum_{n\geq 0}\alpha_n\right]=\sum_{\ci{J\subseteq I}}\alpha_{\ci{J}}=M$ and $\mathbb{E}^\mu\left[\widetilde{f}\right]=\langle f\rangle_{\ci{I}}=\textbf{f}$. Later in subsection 6.2 and subsection 7.2, we use a slightly modified version of this construction.
\end{remark}

\par We will frequently consult to the following proposition.

\begin{proposition}\label{P5}
\begin{enumerate}
\item $\frac{1}{2}(1-\varepsilon)\leq\max\left\{\frac{\mu(\mathcal{X}_{2j-1}^{k+1})}{\mu(\mathcal{X}_j^k)},\frac{\mu(\mathcal{X}_{2j}^{k+1})}{\mu(\mathcal{X}_j^k)}\right\}\leq\frac{1}{2}(1+\varepsilon)$.
\item For every subset $E\in\mathcal{F}_{n_k}$ and $\mu(E\cap\mathcal{X}_j^k)>0$, we have
\begin{align}\label{eq11}
\max\left\{\frac{\mu(E\cap\mathcal{X}_{2j-1}^{k+1})}{\mu(E\cap\mathcal{X}_j^k)}, \frac{\mu(E\cap\mathcal{X}_{2j}^{k+1})}{\mu(E\cap\mathcal{X}_j^k)}\right\}\leq\frac{1}{2}(1+\varepsilon).
\end{align}
Combined with \textup{(i)}, we have
\begin{align}\label{eq12}
\max\left\{\frac{\mu(E\cap\mathcal{X}_{2j-1}^{k+1})}{\mu(\mathcal{X}_{2j-1}^{k+1})}, \frac{\mu(E\cap\mathcal{X}_{2j}^{k+1})}{\mu(\mathcal{X}_{2j}^{k+1})}\right\}\leq\frac{1+\varepsilon}{1-\varepsilon}\cdot\frac{\mu(E\cap\mathcal{X}_{j}^{k})}{\mu(\mathcal{X}_{j}^{k})}.
\end{align}
\item $(1-\varepsilon)^k\leq\frac{\mu(\mathcal{X}_{j}^{k})}{|I_j^k|}\leq(1+\varepsilon)^k$ for all $0\leq k\leq N, 1\leq j\leq2^k$.
\item 
$(1-\varepsilon)^k\langle f\rangle_{\ci{I_j^{N-k}}}\leq \langle\widetilde{f}\rangle_{\ci{X_j^{N-k},\mu}}\leq(1+\varepsilon)^k\langle f\rangle_{\ci{I_j^{N-k}}}$ for all $0\leq k\leq N, 1\leq j\leq2^k$.
\end{enumerate}\label{P}
\end{proposition}
\begin{proof}
\begin{enumerate}
\item This is (\ref{eq10}) from our construction.
\item This is a very important extension of (i). But we only have the upper bound estimation in this general case. Recall that our construction gives $|h_{n_k}|\leq\frac{\varepsilon}{2}$ on $\widetilde{\mathcal{X}_j^k}$, so $\left|\int_{E\cap\widetilde{\mathcal{X}_j^k}}hd\mu\right|\leq\int_{E\cap\widetilde{\mathcal{X}_j^k}}|h_{n_k}|d\mu\leq\frac{\varepsilon}{2}\mu(E\cap\widetilde{\mathcal{X}_j^k})$, which is $\left|\mu(E\cap\mathcal{X}_{2j-1}^{k+1})-\mu(E\cap\mathcal{X}_{2j}^{k+1})\right|\\
\leq\frac{\varepsilon}{2}\mu({E\cap\widetilde{\mathcal{X}_j^k}})\leq\frac{\varepsilon}{2}\mu({E\cap\mathcal{X}_j^k})$. So we obtain (\ref{eq11}). (\ref{eq12}) follows from (\ref{eq11}) and (i).
\item We prove this by induction. For $k=0$, we have $\mu(\mathcal{X}_1^0)=|I_1^0|=1$. Assuming (iii) holds for $k$, by (i) we can estimate for $\mathcal{X}_{2j}^{k+1}$ (same for $\mathcal{X}_{2j-1}^{k+1}$) that
$$(1-\varepsilon)^{k+1}\leq\frac{\mu(\mathcal{X}_{2j}^{k+1})}{|I_{2j}^{k+1}|}=2\cdot\frac{\mu(\mathcal{X}_{2j}^{k+1})}{\mu(\mathcal{X}_j^k)}\cdot\frac{\mu(\mathcal{X}_{j}^{k})}{|I_j^k|}\leq(1+\varepsilon)^{k+1}.$$
\item Again by induction, for $k=0$, since $\widetilde{f}\textbf{1}_{\ci{\mathcal{X}_j^N}}=f\textbf{1}_{\ci{I_j^N}}, 1\leq j\leq 2^N$ and $\widetilde{f}=0$ on $\mathcal{X}\setminus\bigcup_{j=1}^{2^N}\mathcal{X}_j^N$, we have $\langle\widetilde{f}\rangle_{\ci{X_j^{N},\mu}}=\langle f\rangle_{\ci{I_j^N}}$. Assuming (iv) holds for $k$, by (i) we have
\begin{align*}
(1-\varepsilon)^{k+1}\langle f\rangle_{\ci{I_j^{N-(k+1)}}}
& \leq\langle\widetilde{f}\rangle_{\ci{X_j^{N-(k+1)},\mu}}\\
& =\frac{\mu(\mathcal{X}_{2j-1}^{N-k})}{\mu(\mathcal{X}_{j}^{N-(k+1)})}\langle\widetilde{f}\rangle_{\ci{\mathcal{X}_{2j-1}^{N-k},\mu}}+\frac{\mu(\mathcal{X}_{2j}^{N-k})}{\mu(\mathcal{X}_{j}^{N-(k+1)})}\langle\widetilde{f}\rangle_{\ci{\mathcal{X}_{2j}^{N-k},\mu}}\\
& \leq(1+\varepsilon)^{k+1}\langle f\rangle_{\ci{I_j^{N-(k+1)}}}.
\end{align*}
\end{enumerate}
\end{proof}

\section{The Bellman function $\mathcal{B}_\mu^\mathcal{F}(F, \textbf{f}, M;C)$ of Theorem \ref{THM A}}
\subsection{$\mathcal{B}_{\mu}^{\mathcal{F}}(F,\textbf{\textup{f}},M) \leq \mathcal{B}(F,\textbf{\textup{f}},M)$} 
\par We show (\ref{eq 101}) for the case $C=1$ and the general case follows from the scaling property. Take the Bellman function $\mathcal{B}(F,\textbf{f}, M)$ of the dyadic Carleson Embedding Theorem. Consider an arbitrary  function $f\geq 0$ and an arbitrary Carleson sequence $\{\alpha_n\}_{n\geq 0}$ with $C=1$. Set for every $n\geq 0$, 
\begin{align*}
X^n=(F^n, \textbf{f}^n, M^n)=\left(\mathbb{E}^{\mu}\left[f^p|\mathcal{F}_n\right], \mathbb{E}^{\mu}\left[f|\mathcal{F}_n\right], \mathbb{E}^{\mu}\left[\sum_{k\geq n}\alpha_k|\mathcal{F}_n\right]\right).
\end{align*}
Fix the initial step
\begin{align*}
X^0=\left(\mathbb{E}^\mu[f^p], \mathbb{E}^\mu[f], \mathbb{E}^{\mu}\left[\sum_{n\geq 0}\alpha_n\right]\right)=(F,\textbf{f},M).
\end{align*}
By (\ref{eq6}), $0\leq M^n\leq1$, $\textbf{f}^n=f_n$, and when $n\geq1$, $F^n, \textbf{f}^n$ and $M^n$ are random variables. 
\begin{lemma} For every $n\geq0$, we have
$$\mathbb{E}^{\mu}\left[\mathcal{B}(X^n)\right]-\mathbb{E}^{\mu}\left[\mathcal{B}(X^{n+1})\right]\geq\mathbb{E}^{\mu}\left[\alpha_nf_n^p\right],$$
where $\mathcal{B}(X^n)=\mathcal{B}(F^n, \textbf{\textup{f}}^n, M^n)$.
\end{lemma}
\begin{proof}
Recall from subsection 3.2, the Bellman function $\mathcal{B}(F,\textbf{f},M)$ satisfies (\ref{eq 110}). Note also that we have
\begin{align*}
X^n=\mathbb{E}^{\mu}\left[X^{n+1}|\mathcal{F}_n\right]+(0, 0, \alpha_n).
\end{align*}
By Proposition \ref{P1} (v), (\ref{eq 110}) and the Jensen's inequality, we deduce
\begin{align*}
\mathcal{B}(X^n)
&\geq \mathcal{B}\left(\mathbb{E}^{\mu}\left[X^{n+1}|\mathcal{F}_n\right]\right)+\alpha_n f_n^p \geq  \mathbb{E}^{\mu}\left[\mathcal{B}(X^{n+1})|\mathcal{F}_n\right]+\alpha_n f_n^p.
\end{align*}
Taking expectation, we prove exactly
$$\mathbb{E}^{\mu}\left[\mathcal{B}(X^n)\right]-\mathbb{E}^{\mu}\left[\mathcal{B}(X^{n+1})\right]\geq\mathbb{E}^{\mu}\left[\alpha_nf_n^p\right].$$
\end{proof}
Summing up, we get the inequality
\begin{align*}
\mathbb{E}^\mu\left[\sum_{n\geq0}\alpha_nf_n^p\right]\leq\sum_{n\geq0}\left(\mathbb{E}^
\mu[\mathcal{B}(X^n)]-\mathbb{E}^\mu[\mathcal{B}(X^{n+1})]\right)\leq\mathcal{B}(X^0).
\end{align*}
Hence, we conclude that $\mathcal{B}_{\mu}^{\mathcal{F}}(F,\textbf{\textup{f}},M) \leq \mathcal{B}(F,\textbf{\textup{f}},M)$.

\subsection{$\mathcal{B}_{\mu}^{\mathcal{F}}(F,\textbf{\textup{f}},M)=\mathcal{B}(F,\textbf{\textup{f}},M)$ for an infinitely refining filtration}
\par To show (\ref{eq 102A}), again we consider $C=1$. Note first that on the boundary $\textbf{f}^p=F$, we have $\mathcal{B}_{\mu}^{\mathcal{F}}(F,\textbf{\textup{f}},M)=\mathcal{B}(F,\textbf{\textup{f}},M)=MF$. For the case $\textbf{f}^p<F$, we need to apply the remodeling from section 5.

\par For technical issues, we slightly modify our remodeling here. First, by the continuity of $\mathcal{B}$ (see subsection 3.2), there exists $\delta_1>0$, such that $\textbf{f}^p<F-\delta_1$ and $\mathcal{B}(F-\delta_1,\textbf{\textup{f}},M)$ is close to $\mathcal{B}(F,\textbf{\textup{f}},M)$. Next, by the definition of $\mathcal{B}$, we can find a non-negative function $f$ on the unit interval $I=[0, 1]$ with $\langle f^p\rangle_{\ci{I}}=F-\delta_1$, $\langle f\rangle_{\ci{I}}=\textbf{f}$ and a sequence $\{\alpha_{\ci{J}}\}_{\ci{J\subseteq I}}$, $\sum_{\ci{J\subseteq{I}}}\alpha_{\ci{J}}=M$ satisfying the Carleson condition with constant $C=1$, such that the sum $\sum_{\ci{J\subseteq I}}\alpha_{\ci{J}}\langle f\rangle_{\ci{J}}^p$ (almost) equals $\mathcal{B}(F,\textbf{f},M)$. Moreover, by again the continuity, we can choose a \emph{finite} subset of $\{\alpha_{\ci{J}}\}_{\ci{J\subseteq I}}$ such that $\sum_{\ci{J\subseteq{I}}}\alpha_{\ci{J}}=M-\delta_2$ for some $\delta_2>0$ and $\sum_{\ci{J\subseteq I}}\alpha_{\ci{J}}\langle f\rangle_{\ci{J}}^p$ still (almost) equals $\mathcal{B}(F,\textbf{f},M)$. For simplicity, we assume exactly

\begin{align}\label{eq last}
\sum_{\ci{J\subseteq I}}\alpha_{\ci{J}}\langle f\rangle_{\ci{J}}^p=\mathcal{B}(F,\textbf{f},M).
\end{align}

Let the indices of $\{\alpha_{\ci{J}}\}_{\ci{J\subseteq I}}$ belong to the collection $\{I_j^k: 1\leq k\leq N, 1\leq j\leq2^k\}$ for some fixed integer $N$. Choose $\varepsilon>0$, such that $F-\delta_1\leq\frac{1}{(1+\varepsilon)^N}F$. We do the remodeling with this $\varepsilon>0$ to construct $\{\mathcal{X}_j^k: 0\leq k\leq N, 1\leq j\leq2^k\}$, $\{\alpha_n\}_{n\geq0}$ and $\widetilde{f}$ on the space  $(\mathcal{X},\mathcal{F},\mu)$. To proceed, we observe that

\begin{lemma}
\begin{align}
\mathbb{E}^\mu\left[\widetilde{f}^p\right]\leq(1+\varepsilon)^N\langle f^p\rangle_{\ci{I}}.\label{c3}
\end{align}
\end{lemma}
\begin{proof}
By (iii) of Proposition \ref{P},
$$\mathbb{E}^\mu\left[\widetilde{f}^p\right]=\sum_{j=1}^{2^N}\langle\widetilde{f}^p\rangle_{\ci{\mathcal{X}_{j}^N,\mu}}\mu(\mathcal{X}_j^N)\leq\sum_{j=1}^{2^N}\langle f^p\rangle_{\ci{I_j^N}}\cdot (1+\varepsilon)^N |I_j^N|=(1+\varepsilon)^N\langle f^p\rangle_{\ci{I}}.$$
\end{proof}

So (\ref{c3}) and $\langle f^p\rangle_{\ci{I}}=F-\delta_1\leq\frac{1}{(1+\varepsilon)^N}F$ imply that  $\mathbb{E}^\mu\left[\widetilde{f}^p\right]\leq F$. Also recall from the remodeling, we know $\mathbb{E}^\mu\left[\widetilde{f}\right]=\langle f\rangle_{\ci{I}}=\textbf{f}$. Let us further modify the function $\widetilde{f}$ in the following way. Note that we are working on an infinitely refining filtration (see definition \ref{Def2}). There exists a simple function $g$ behaving like a Haar function, such that $g$ is supported on $\mathcal{X}_1^N$, $\langle g\rangle_{\ci{\mathcal{X}_1^N,\mu}}=0$ and $0<\mathbb{E}^\mu\left[|g|^p\right]<\infty$. Consider the continuous function 
$$a(t)=\mathbb{E}^\mu\left[\left|\widetilde{f}+tg\right|^p\right].$$
Thus, $a(0)\leq F$ and $\lim_{t\rightarrow\infty}a(t)=\infty$. Hence, we can find $t_0\geq 0$, such that $\mathbb{E}^\mu\left[\left|\widetilde{f}+t_0g\right|^p\right]=F$. Update $\widetilde{f}$ to $\widetilde{f}+t_0g$. We have then $\mathbb{E}^\mu\left[\left|\widetilde{f}\right|^p\right]=F$ and $\mathbb{E}^\mu\left[\widetilde{f}\right]=\textbf{f}$. Note here the updated function $\widetilde{f}$ might be negative, however, all the relevant average values we will use are still non-negative.

\par Now, let us discuss the properties of the Carleson sequence $\{\alpha_n\}_{n\geq0}$. Directly from the remodeling, we know $\mathbb{E}^\mu\left[\sum_{n\geq 0}\alpha_n\right]=\sum_{\ci{J\subseteq{I}}}\alpha_{\ci{J}}=M-\delta_2$. Moreover, we can prove

\begin{lemma}
The non-negative sequence $\{\alpha_n\}_{n\geq0}$ satisfies each $\alpha_n$ is $\mathcal{F}_n$-measurable and
\begin{align}
\mathbb{E}^\mu\left[\sum_{k\geq n}\alpha_k|\mathcal{F}_n\right]\leq\frac{(1+\varepsilon)^N}{(1-\varepsilon)^{2N}}~\textup{for every}~n\geq0.\label{c1}
\end{align}
\end{lemma}
\begin{proof}
From the construction, it is clear that each $\alpha_n$ is non-negative and $\mathcal{F}_n$-measurable. So we need to show for every $\mathcal{F}_n$-measurable set $E$, we have
\begin{align*}
\mathbb{E}^\mu\left[\sum_{k\geq n}\alpha_k\textbf{1}_{\ci{E}}\right]\leq\frac{(1+\varepsilon)^N}{(1-\varepsilon)^{2N}}\cdot\mu(E).
\end{align*}
Denote by $k_0=\min\{k: n_k\geq n\}$. Since $\mathbb{E}^\mu\left[\sum_{k\geq n}\alpha_k\textbf{1}_{\ci{E}}\right]=\mathbb{E}^\mu\left[\mathbb{E}^\mu\left[\sum_{k\geq k_0}\alpha_{n_k}|\mathcal{F}_{n_{k_0}}\right]\textbf{1}_{\ci{E}}\right]$, it suffices to show
\begin{align*}
\mathbb{E}^\mu\left[\sum_{k\geq k_0}\alpha_{n_k}|\mathcal{F}_{n_{k_0}}\right]\leq\frac{(1+\varepsilon)^N}{(1-\varepsilon)^{2N}}, 
\end{align*}
or equivalently, for every $\mathcal{F}_{n_{k_0}}$-measurable set $E$, we have
\begin{align*}
\mathbb{E}^\mu\left[\sum_{k\geq k_0}\alpha_{n_k}\textbf{1}_{\ci{E}}\right]\leq\frac{(1+\varepsilon)^N}{(1-\varepsilon)^{2N}}\cdot\mu(E). 
\end{align*}
Now the explicit computation shows
$$\mathbb{E}^\mu\left[\sum_{k\geq k_0}\alpha_{n_k}\textbf{1}_{\ci{E}}\right]=\sum_{k\geq k_0}\sum_{j=1}^{2^k}\alpha_{\ci{I_j^k}}\frac{\mu(E\cap\mathcal{X}_j^k)}{\mu(\mathcal{X}_j^k)}.$$
An iteration of (\ref{eq12}) gives
\begin{align*}
\frac{\mu(E\cap\mathcal{X}_j^k)}{\mu(\mathcal{X}_j^k)}\leq\frac{(1+\varepsilon)^N}{(1-\varepsilon)^N}\cdot\frac{\mu(E\cap\mathcal{X}_l^{k_0})}{\mu(\mathcal{X}_l^{k_0})},~\textup{whenever}~
\mathcal{X}_j^k\subseteq\mathcal{X}_l^{k_0}.
\end{align*}
So we can estimate
\begin{align*}
\mathbb{E}^\mu\left[\sum_{k\geq k_0}\alpha_{n_k}\textbf{1}_{\ci{E}}\right]
&\leq \sum_{l=1}^{2^{k_0}}\frac{(1+\varepsilon)^N}{(1-\varepsilon)^N}\cdot\frac{\mu(E\cap\mathcal{X}_l^{k_0})}{\mu(\mathcal{X}_l^{k_0})}\sum_{k, j: \mathcal{X}_j^k\subseteq\mathcal{X}_l^{k_0}}\alpha_{\ci{I_j^k}},~(\{\alpha_{\ci{I}}\}~\textup{is a Carleson sequence})\\
&\leq \frac{(1+\varepsilon)^N}{(1-\varepsilon)^N}\sum_{l=1}^{2^{k_0}}\frac{\mu(E\cap\mathcal{X}_l^{k_0})}{\mu(\mathcal{X}_l^{k_0})}\cdot\left|I_l^{k_0}\right|,~(\textup{Proposition \ref{P} (iii)})\\
&\leq \frac{(1+\varepsilon)^N}{(1-\varepsilon)^{2N}}\sum_{l=1}^{2^{k_0}}\mu(E\cap\mathcal{X}_l^{k_0})\leq \frac{(1+\varepsilon)^N}{(1-\varepsilon)^{2N}}\cdot\mu(E). 
\end{align*}
\end{proof}

\par To finish, we need one final lemma.

\begin{lemma}
\begin{align}
\mathbb{E}^\mu\left[\sum_{n\geq0}\alpha_n\left|\widetilde{f}_n\right|^p\right]\geq(1-\varepsilon)^{pN}
\sum_{J\subseteq I}\alpha_{\ci{J}}\langle f\rangle_{\ci{J}}^p.\label{c2}
\end{align}
\end{lemma}
\begin{proof}
\begin{align*}
\mathbb{E}^\mu\left[\sum_{n\geq0}\alpha_n\left|\widetilde{f}_n\right|^p\right]
&=\mathbb{E}^\mu\left[\sum_{k\geq0}\alpha_{n_k}\left|\widetilde{f}_{n_k}\right|^p\right]
=\sum_{k\geq0}\sum_{j=1}^{2^k}\alpha_{I_j^k}
\left\langle\left|\widetilde{f}_{n_k}\right|^p\right\rangle_{\ci{\mathcal{X}_j^k,\mu}}\\
&\geq\sum_{k\geq0}\sum_{j=1}^{2^k}\alpha_{\ci{I_j^k}}
\langle\widetilde{f}_{n_k}\rangle^p_{\ci{\mathcal{X}_j^k,\mu}}=\sum_{k\geq0}\sum_{j=1}^{2^k}\alpha_{\ci{I_j^k}}
\langle\widetilde{f}\rangle^p_{\ci{\mathcal{X}_j^k,\mu}},~\textup{(Proposition \ref{P} (iv))}\\
& \geq(1-\varepsilon)^{pN}\sum_{k\geq0}\sum_{j=1}^{2^k}\alpha_{I_j^k}
\langle f\rangle^p_{\ci{I_j^k}}
=(1-\varepsilon)^{pN}\sum_{J\subseteq I}\alpha_{\ci{J}}\langle f\rangle_{\ci{J}}^p.
\end{align*}
\end{proof}

Summarizing, we have constructed a function $\widetilde{f}$ and a Carleson sequence $\{\alpha_n\}_{n\geq0}$ satisfying (\ref{c1}) with $\mathbb{E}^\mu\left[\left|\widetilde{f}\right|^p\right]=F$,  $\mathbb{E}^\mu\left[\widetilde{f}\right]=\textbf{f}$ and $\mathbb{E}^\mu\left[\sum_{n\geq 0}\alpha_n\right]=\sum_{\ci{J\subseteq{I}}}\alpha_{\ci{J}}=M-\delta_2$. By (\ref{eq last}) and (\ref{c2}), we deduce 

\begin{align*}
\mathcal{B}_\mu^{\mathcal{F}}\left(F, \textbf{f}, M-\delta_2; C=\frac{(1+\varepsilon)^N}{(1-\varepsilon)^{2N}}\right)
& \geq (1-\varepsilon)^{pN}
\sum_{J\subseteq I}\alpha_{\ci{J}}\langle f\rangle_{\ci{J}}^p =(1-\varepsilon)^{pN}\mathcal{B}\left(F, \textbf{f}, M\right).
\end{align*}

And Proposition \ref{P1} (iv) and (v) imply that

\begin{align*}
\mathcal{B}_\mu^{\mathcal{F}}\left(F, \textbf{f}, M-\delta_2; C=\frac{(1+\varepsilon)^N}{(1-\varepsilon)^{2N}}\right)
& = \frac{(1+\varepsilon)^N}{(1-\varepsilon)^{2N}}\mathcal{B}_\mu^{\mathcal{F}}\left(F, \textbf{f}, \frac{(1-\varepsilon)^{2N}}{(1+\varepsilon)^N}(M-\delta_2)\right)\\
&\leq  \frac{(1+\varepsilon)^N}{(1-\varepsilon)^{2N}}\mathcal{B}_\mu^{\mathcal{F}}\left(F, \textbf{f},M\right).
\end{align*}

Letting $\varepsilon\rightarrow 0$, the continuity of $\mathcal{B}$ gives exactly $\mathcal{B}_{\mu}^{\mathcal{F}}(F,\textbf{\textup{f}},M)\geq\mathcal{B}(F,\textbf{\textup{f}},M)$. The other inequality is proved in the subsection 6.1.

\section{The Bellman function $\widetilde{\mathcal{B}}_\mu^\mathcal{F}(F, \textbf{f})$ of the maximal operators}

\subsection{$\widetilde{\mathcal{B}}_{\mu}^{\mathcal{F}}(F,\textbf{\textup{f}}) \leq \mathcal{B}_{\mu}^{\mathcal{F}}(F,\textbf{\textup{f}},1)$} 
Let us relate the maximal function (\ref{eq13}) to the Bellman function $\mathcal{B}_{\mu}^{\mathcal{F}}(F,\textbf{\textup{f}},M)$. Define $E_n=\{x\in\mathcal{X}: n~\textup{is the smallest integer, such that}~f^*(x)=|f_n(x)|\}$. Obviously, $\{E_n\}_{n\geq0}$ forms a disjoint partition of $\mathcal{X}$. We can compute
\begin{align*}
||f^*||_{\ci{L^p(\mathcal{X},\mathcal{F},\mu)}}^p
& =\mathbb{E}^\mu\left[|f^*|^p\right]=\mathbb{E}^\mu\left[\sum_{n\geq0}|f_n|^p\textbf{1}_{\ci{E_n}}\right]\\
&=\mathbb{E}^\mu\left[\sum_{n\geq0}\mathbb{E}^\mu[\textbf{1}_{\ci{E_n}}|\mathcal{F}_n]\cdot|f_n|^p\right].
\end{align*}
Let $\alpha_n=\mathbb{E}^\mu[\textbf{1}_{\ci{E_n}}|\mathcal{F}_n]$, $n\geq0$. The connection between the maximal function (\ref{eq13}) and $\mathcal{B}_{\mu}^{\mathcal{F}}(F,\textbf{\textup{f}},M)$ relies on the following simple fact.
\begin{lemma}
$\{\alpha_n\}_{n\geq0}$ is a Carleson sequence with $C=1$. (see Definition \ref{Def1}).
\end{lemma} 
\begin{proof}
It is clear that each $\alpha_n$ is non-negative and $\mathcal{F}_n$-measurable. Moreover, for every set $E\in\mathcal{F}_n$, we have $\mathbb{E}^\mu\left[\sum_{k\geq n}\alpha_k\textbf{1}_{\ci{E}}\right]=\mathbb{E}^\mu\left[\sum_{k\geq n}\textbf{1}_{\ci{E_k\cap E}}\right]\leq\mu(E)$. So we prove the claim.
\end{proof}
To prove (\ref{eq 1001}), fix $\mathbb{E}^\mu[f^p]=F$ and $\mathbb{E}^\mu[f]=\textbf{f}$. Since $\{\alpha_n\}_{n\geq0}$ is a Carleson sequence with $C=1$ and $\mathbb{E}^\mu\left[\sum_{n\geq 0}\alpha_n\right]=1$, we conclude that $\widetilde{\mathcal{B}}_{\mu}^{\mathcal{F}}(F,\textbf{\textup{f}}) \leq \mathcal{B}_{\mu}^{\mathcal{F}}(F,\textbf{\textup{f}},1)$.

\subsection{$\widetilde{\mathcal{B}}_{\mu}^{\mathcal{F}}(F,\textbf{\textup{f}}) = \mathcal{B}_{\mu}^{\mathcal{F}}(F,\textbf{\textup{f}},1)$ for an infinitely refining filtration} 
Again, we appeal to the modified remodeling from subsection 6.2, but only with $M=1$.  Note that we have $$\mathbb{E}^\mu\left[\sum_{n\geq0}\alpha_n\left|\widetilde{f}_n\right|^p\right]
=\mathbb{E}^\mu\left[\sum_{k\geq0}\alpha_{n_k}\left|\widetilde{f}_{n_k}\right|^p\right]
=\sum_{k\geq0}\sum_{j=1}^{2^k}\alpha_{I_j^k}
\left\langle\left|\widetilde{f}_{n_k}\right|^p\right\rangle_{\ci{\mathcal{X}_j^k,\mu}}.$$ 
To proceed, we observe that

\begin{lemma} For every $0\leq k\leq N$ and $1\leq j\leq2^k$, we have
\begin{align}
\left|\widetilde{f}_{n_{\ci{N-k}}}\right|\textbf{\textup{1}}_{\ci{\mathcal{X}_j^{N-k}}}\leq\frac{(1+\varepsilon)^k}{(1-\varepsilon)^k}\langle\widetilde{f}\rangle_{\ci{\mathcal{X}_j^{N-k},\mu}}.\label{eq15}
\end{align}
\end{lemma}
\begin{proof}
First note that $\left|\widetilde{f}_{n_{\ci{N-k}}}\right|\textbf{\textup{1}}_{\ci{\mathcal{X}_j^{N-k}}} = \widetilde{f}_{n_{\ci{N-k}}}\textbf{\textup{1}}_{\ci{\mathcal{X}_j^{N-k}}}$ for every $0\leq k\leq N$ and $1\leq j\leq2^k$. Induction on k, for $k=0$, the construction of $\widetilde{f}$ immediately gives $\widetilde{f}_{n_{\ci{N}}}\textbf{1}_{\ci{\mathcal{X}_j^N}}=
\langle\widetilde{f}\rangle_{\ci{\mathcal{X}_j^{N},\mu}}$, $1\leq j\leq2^N$. Assuming (\ref{eq15}) holds for $k$, then for every $\mathcal{F}_{n_{\ci{N-(k+1)}}}$-measurable set $E$, $E\subseteq\mathcal{X}_j^{N-(k+1)}$ and $\mu(E)>0$, we can estimate
\begin{align*}
\int_E\widetilde{f}_{n_{\ci{N-(k+1)}}}\textbf{1}_{\ci{\mathcal{X}_j^{N-(k+1)}}}d\mu
& =\int_{E\cap\mathcal{X}_j^{N-(k+1)}}\widetilde{f}d\mu
=\int_{E\cap\mathcal{X}_{2j-1}^{N-k}}\widetilde{f}_{n_{\ci{N-k}}}d\mu
+\int_{E\cap\mathcal{X}_{2j}^{N-k}}\widetilde{f}_{n_{\ci{N-k}}}d\mu\\
& \leq\frac{(1+\varepsilon)^k}{(1-\varepsilon)^k}\left[\langle\widetilde{f}\rangle_{\ci{\mathcal{X}_{2j-1}^{N-k},\mu}}\mu(E\cap\mathcal{X}_{2j-1}^{N-k})+\langle\widetilde{f}\rangle_{\ci{\mathcal{X}_{2j}^{N-k},\mu}}\mu(E\cap\mathcal{X}_{2j}^{N-k})\right].
\end{align*}
And hence, we can deduce
\begin{align*}
&~\mu(E)^{-1}\int_E\widetilde{f}_{n_{\ci{N-(k+1)}}}\textbf{1}_{\ci{\mathcal{X}_j^{N-(k+1)}}}d\mu,
~\textup{(}E\subseteq\mathcal{X}_{j}^{N-(k+1)}\textup{)}\\
&\leq\frac{(1+\varepsilon)^k}{(1-\varepsilon)^k}\left[ \langle\widetilde{f}\rangle_{\ci{\mathcal{X}_{2j-1}^{N-k},\mu}}\frac{\mu(E\cap\mathcal{X}_{2j-1}^{N-k})}{\mu(E\cap\mathcal{X}_{j}^{N-{k+1}})}+\langle\widetilde{f}\rangle_{\ci{\mathcal{X}_{2j}^{N-k},\mu}}\frac{\mu(E\cap\mathcal{X}_{2j}^{N-k})}{\mu(E\cap\mathcal{X}_{j}^{N-{k+1}})}\right],~\textup{(\ref{eq11})}\\
&\leq\frac{1}{2}\frac{(1+\varepsilon)^{k+1}}{(1-\varepsilon)^k}\left[\langle\widetilde{f}\rangle_{\ci{\mathcal{X}_{2j-1}^{N-k},\mu}}+\langle\widetilde{f}\rangle_{\ci{\mathcal{X}_{2j}^{N-k},\mu}}\right],~\textup{(Proposition \ref{P} (i))}\\
& \leq\frac{(1+\varepsilon)^{k+1}}{(1-\varepsilon)^{k+1}}\langle\widetilde{f}\rangle_{\ci{\mathcal{X}_j^{N-(k+1)},\mu}}.
\end{align*}
Since this is true for every $\mathcal{F}_{n_{\ci{N-(k+1)}}}$-measurable set $E$, $E\subseteq\mathcal{X}_j^{N-(k+1)}$ and $\mu(E)>0$, we prove (\ref{eq15}) for $k+1$.
\end{proof}

Applying (\ref{eq15}), we have
$$\mathbb{E}^\mu\left[\sum_{n\geq0}\alpha_n\left|\widetilde{f}_n\right|^p\right]
=\sum_{k\geq0}\sum_{j=1}^{2^k}\alpha_{I_j^k}
\left\langle\left|\widetilde{f}_{n_k}\right|^p\right\rangle_{\ci{\mathcal{X}_j^k,\mu}}\leq\frac{(1+\varepsilon)^{pN}}{(1-\varepsilon)^{pN}}\sum_{k\geq0}\sum_{j=1}^{2^k}\alpha_{I_j^k}
\langle\widetilde{f}\rangle_{\ci{\mathcal{X}_j^k,\mu}}^p.$$
And note that Proposition \ref{P} (iii) implies
$$\sum_{k,j: I_j^k\subseteq I_{j_0}^{k_0}}\alpha_{\ci{I_j^k}}\leq\left|I_{j_0}^{k_0}\right|\leq\frac{1}{(1-\varepsilon)^N}\mu(\mathcal{X}_{j_0}^{k_0})~\textup{for every}~0\leq k_0\leq N, 1\leq j_0\leq2^{k_0}.$$
Now, let us recall a useful lemma established in \cite{M}, formulated in our language, 

\begin{lemma}\label{l}
Suppose $\alpha_{\ci{I_j^k}}\geq0$, where $0\leq k\leq N, 1\leq j\leq2^k$, satisfies
\begin{align}\label{eq16}
\sum_{k,j: I_j^k\subseteq I_{j_0}^{k_0}}\alpha_{\ci{I_j^k}}\leq C\mu(\mathcal{X}_{j_0}^{k_0})
\end{align}
for some constant $C>0$, then we can choose pairwise disjoint measurable $\mathcal{A}_j^k\subseteq\mathcal{X}$ such that $\mathcal{A}_j^k\subseteq\mathcal{X}_j^k$ and $\alpha_{\ci{I_j^k}}=C\mu(\mathcal{A}_j^k)$.
\end{lemma}

\begin{proof}
Without loss of generality, we can assume $C=1$. We start at the level $k=N$. Since (\ref{eq16}) with $C=1$ implies $\alpha_{\ci{I_j^N}}\leq\mu(\mathcal{X}_{j}^{N})$ for every $1\leq j\leq2^N$, we can choose $\mathcal{A}_j^N\subseteq\mathcal{X}_j^N$ such that $\alpha_{\ci{I_j^N}}=\mu(\mathcal{A}_j^N)$. Assuming that we have chosen pairwise disjoint measurable $\mathcal{A}_j^k$ for all $k\geq k_0+1$ and $1\leq j\leq2^k$, note that (\ref{eq16}) with $C=1$ gives
$$\alpha_{\ci{I_{j_0}^{k_0}}}+\sum_{k,j: I_j^k\subsetneqq I_{j_0}^{k_0}}\alpha_{\ci{I_j^k}}\leq \mu(\mathcal{X}_{j_0}^{k_0}),~~\textup{so}~~\alpha_{\ci{I_{j_0}^{k_0}}}\leq
\mu\left(\mathcal{X}_{j_0}^{k_0}\setminus\bigcup_{k,j: I_j^k\subsetneqq I_{j_0}^{k_0}}\mathcal{A}_j^k\right),$$
and thus we can choose measurable set $\mathcal{A}_{j_0}^{k_0}\subseteq\mathcal{X}_{j_0}^{k_0}\setminus\bigcup_{k,j: I_j^k\subsetneqq I_{j_0}^{k_0}}\mathcal{A}_j^k$, such that $\alpha_{\ci{I_{j_0}^{k_0}}}=\mu(\mathcal{A}_{j_0}^{k_0})$. Continue this process for all the indices. This proves the lemma.
\end{proof}

By Lemma \ref{l}, we can estimate
\begin{align*}
\mathbb{E}^\mu\left[\sum_{n\geq0}\alpha_n\left|\widetilde{f}_n\right|^p\right]
& \leq\frac{(1+\varepsilon)^{pN}}{(1-\varepsilon)^{pN}}\sum_{k\geq0}\sum_{j=1}^{2^k}\alpha_{I_j^k}
\langle\widetilde{f}\rangle_{\ci{\mathcal{X}_j^k,\mu}}^p\\
& = \frac{(1+\varepsilon)^{pN}}{(1-\varepsilon)^{pN}}\sum_{k\geq0}\sum_{j=1}^{2^k}\frac{1}{(1-\varepsilon)^N}\mu(\mathcal{A}_j^k)
\langle\widetilde{f}_{n_k}\rangle_{\ci{\mathcal{X}_j^k,\mu}}^p\\
& \leq\frac{(1+\varepsilon)^{pN}}{(1-\varepsilon)^{(p+1)N}}\sum_{k\geq0}\sum_{j=1}^{2^k}\mathbb{E}^\mu\left[\left|\widetilde{f}^*\right|^p\textbf{1}_{\ci{\mathcal{A}_j^k}}\right],~\textup{(disjointness)}\\
& \leq \frac{(1+\varepsilon)^{pN}}{(1-\varepsilon)^{(p+1)N}}\mathbb{E}^\mu\left[\left|\widetilde{f}^*\right|^p\right].
\end{align*}

Applying (\ref{eq last}) and (\ref{c2}) with $M=1$, together with Theorem \ref{THM A}, we have
\begin{align*}
\frac{(1+\varepsilon)^{pN}}{(1-\varepsilon)^{(p+1)N}}\mathbb{E}^\mu\left[\left|\widetilde{f}^*\right|^p\right]
& \geq \mathbb{E}^\mu\left[\sum_{n\geq0}\alpha_n\left|\widetilde{f}_n\right|^p\right]
\geq (1-\varepsilon)^{pN}
\sum_{J\subseteq I}\alpha_{\ci{J}}\langle f\rangle_{\ci{J}}^p\\
& =(1-\varepsilon)^{pN}\mathcal{B}\left(F, \textbf{f}, 1\right)=(1-\varepsilon)^{pN}\mathcal{B}_\mu^{\mathcal{F}}\left(F, \textbf{f}, 1\right).
\end{align*}

Recall that $\mathbb{E}^\mu\left[\left|\widetilde{f}\right|^p\right]=F$ and $\mathbb{E}^\mu\left[\widetilde{f}\right]=\textbf{f}$. 
Letting $\varepsilon\rightarrow 0$, the continuity of $\mathcal{B}$, and thus of $\mathcal{B}_\mu^{\mathcal{F}}$, gives exactly $\widetilde{\mathcal{B}}_{\mu}^{\mathcal{F}}(F,\textbf{\textup{f}}) \geq \mathcal{B}_{\mu}^{\mathcal{F}}(F,\textbf{\textup{f}},1)$. The other inequality is proved in the subsection 7.1.

\section*{Acknowledgement}
The author would like to thank his PhD thesis advisor, Serguei Treil, for suggesting this problem and for the invaluable guidance and support throughout the course of preparation.

\end{document}